\documentclass[12pt]{amsart}
\usepackage{amsthm}
\usepackage{amstext}
\usepackage{amssymb}
\usepackage{mathrsfs}
\usepackage{mathabx}
\usepackage[english]{babel}

\usepackage{enumerate}
\usepackage{xfrac}
\usepackage{mathabx}
\usepackage{mathtools}
\usepackage{xcolor}
\usepackage[all]{xy}

\usepackage[T1]{fontenc}

\theoremstyle{plain}
\newtheorem{theorem}{Theorem}[section]

\newtheorem{lemma}[theorem]{Lemma}

\theoremstyle{definition}
\newtheorem{definition}[theorem]{Definition}

\newtheorem{example}[theorem]{Example}
\newtheorem{remark}[theorem]{Remark}

\theoremstyle{remark}

\numberwithin{equation}{section}

\newcommand{\R}{\mathbb R}

\newcommand{\Id}{\mathrm{d}}

\newcommand{\IR}{\mathbb{R}}

\def\bf{\mathbf}

\newcommand{\loc}{\mathrm{loc}}

\newcommand{\IB}{B}

\newcommand{\dom}{\mathrm{Dom}}

\newcommand{\IN}{\mathbb{N}}

\newcommand{\e}{e}

\newcommand\newdot{{\kern.8pt\cdot\kern.8pt}}
\def\nbull{{\raise1.5pt\hbox{\bf .}}}

\title[Neumann cut-offs and essential self-adjointness on manifolds with boundary]{Neumann cut-offs and essential self-adjointness on complete Riemannian manifolds with boundary}%

\author{Davide Bianchi}
\address{School of Mathematics (Zhuhai), Sun Yat-sen University\\
Zhuhai campus, Haiqin Building No.~2, 519082 Zhuhai, China}
\email{bianchid@mail.sysu.edu.cn}
\author{Batu Güneysu}
\address{Fakultät für Mathematik, Technische Universität Chemnitz\\
Reichenhainer Straße 41, 09126 Chemnitz, Germany}
\email{batu.gueneysu@math.tu-chemnitz.de}
\author{Alberto G. Setti}
\address{DiSAT, Università dell'Insubria, Como-Varese\\
via Valleggio 11, 22100 Como, Italy}
\email{alberto.setti@uninsubria.it}
\thanks{The first  author is supported by the Startup Fund of Sun Yat-sen University. The third  author is member of the GNAMPA INdAM group. }

\begin{document}

\begin{abstract}We generalize some fundamental results for noncompact Riemannian manfolds without boundary, that only require completeness and no curvature assumptions, to manifolds with boundary: let $M$ be a smooth Riemannian manifold with boundary $\partial M$ and let $\hat{C}^\infty_c(M)$ denote the space of smooth compactly supported cut-off functions with vanishing normal derivative, \emph{Neumann cut-offs}. We show, among other things, that under completeness:
	\begin{itemize}
		\item $\hat{C}^\infty_c(M)$ is dense in $W^{1,p}(\mathring{M})$ for all $p\in (1,\infty)$; this generalizes a classical result by Aubin~\cite{aubin} for $\partial M=\emptyset$.
		\item $M$ admits a sequence of first order cut-off functions in $\hat{C}^\infty_c(M)$; for $\partial M=\emptyset$ this result can be traced back to Gaffney~\cite{gaffney}.
		
		\item the Laplace-Beltrami operator with domain of definition $\hat{C}^\infty_c(M)$ is essentially self-adjoint; this is a generalization of a classical result by Strichartz~\cite{strichartz} for $\partial M=\emptyset$.
		
	\end{itemize}

\end{abstract}

\maketitle

\section{First order Sobolev density results, and first order sequences of Neumann cut-off functions}

Let $M$ be a smooth connected Riemannian manifold with boundary with dimension $m$, where if there is no danger of confusion we will denote the Riemannian metric $g$ on $M$ simply with $(\cdot,\cdot)$. We consider $M$ as a metric space with respect to its geodesic distance $\varrho(x,y):=\inf L(\gamma)$, where the infimum is taken over all piecewise smooth curves $\gamma:[0,1]\to M$ joining $x$ and $y$, and where $L(\gamma):=\int^b_a |\dot{\gamma}(t)|\Id t$ denotes the lenght of such a curve. The induced open balls are denoted with $B(x,r)$. Note that $\varrho$ induces the original topology on $M$ \cite{alexander}. The completeness of $M$ will always mean its completeness as a metric space, and by the Hopf-Rinow Theorem for locally compact lenght spaces, this is equivalent to bounded sets being relatively compact.

\begin{remark}
	Assume $M'$ is an open connected smooth subset of a smooth complete Riemannian manifold without boundary. Then $M:=\overline{M'}$ is a complete smooth Riemannian manifold with boundary. This follows from the fact that the (intrinsic!) geodesic distance on $M$ induces the topology on $M$.
\end{remark}

We denote by $\mathring{M}  = M\backslash \partial M$ the interior of $M$ in the sense of manifolds with boundary. Following the notation of \cite{Grigoryan_laplace_eq} and \cite{IPS_Crelle}, for any open set $\Omega$ in $M$ we set $\mathring{\Omega} = \Omega \cap \mathring{M}$, $\partial\Omega= \partial \mathring{\Omega}$ and let $\partial_0 \Omega = \partial \Omega \cap \mathring{M}$ be the Dirichlet boundary of $\Omega$, and $\partial_1 \Omega =  \bar{\Omega} \cap \partial M$ be the Neumann boundary of $\Omega$, so that $\partial\Omega= \partial_0\Omega\cup \partial_1\Omega$.\vspace{1mm}

The symbol $\partial_\nu$ denotes the outward pointing normal derivative on $\partial M$. \vspace{1mm}

We understand all $L^p$-norms and $L^2$-scalar products with respect to the Riemannian volume measure $\mu$ given locally by $\mu(\Id x)=  \sqrt{g}(x) \Id x$, where of course $\mu(M\setminus \mathring{M})=\mu(\partial M)=0$, as $\partial M$ is assumed to be smooth.\vspace{1mm}

Let $\nabla$ denote the Levi-Civita connection on $M$, and let $\Delta$ denote the (negative definite) Laplace-Beltrami operator, so that we have Green's formula
$$
\left\langle \Delta f_1, f_2 \right\rangle = -\left\langle \nabla f_1\nabla f_2\right\rangle + \int_{\partial M} (\partial_\nu f_1)  f_2  \    \Id\sigma,
$$
valid for all $f_1,f_2\in C^\infty(M)$ one of which having a compact support, where $\sigma$ is the $(m-1)$-dimensional Hausdorff measure of the metric space $M$. \vspace{1mm}

We recall that for every smooth Riemannian manifold without boundary $N$ the Banach space one defines
$$
W^{1,p}(N)=\{f\in L^p(N): \nabla f\in \Gamma_{L^p}(N,TN)\},
$$
where expressions of the form $Pf$ with $f\in L^1_\loc(N)$ and $P$ a differential operator with smooth coefficients in $N$, are understood in the usual sense of distribution theory of open manifolds. Then $W^{1,p}(N)$ becomes a Banach space with respect to the norm $\left\|f\right\|_p+\left\|\nabla f\right\|_p$. \vspace{1mm}

It has been shown in \cite{giona}, Corollary 2.3, that if $M$ is complete then $C^{\infty}_c(M)$ is dense in $W^{1,p}(\mathring{M})$.\vspace{1mm}

The Neumann realization of $-\Delta$ in $L^2(M)$ is defined as the nonnegative self-adjoint operator $H$ in $L^2(M)$ associated with the regular, strongly local Dirichlet form
$$
\mathscr{E}(f_1,f_2):=\left\langle \nabla f_1,\nabla f_2\right\rangle= \int_M (\nabla f_1,\nabla f_2)\Id \mu,
$$
with $\dom(\mathscr{E})=W^{1,2}(\mathring{M})$. In other words, $H$ is the uniquely determined nonnegative self-adjoint operator in $L^2(M)$ with $\dom(H)\subset W^{1,2}(\mathring{M})$ and
\[
\left\langle Hf_1,f_2\right\rangle = \left\langle \nabla f_1,\nabla f_2 \right\rangle \quad\text{ for all $f_1\in \dom(H)$, $f_2\in W^{1,2}(\mathring{M})$.}
\]
In particular, for all $f\in \dom(H)$ one has $\Delta f\in L^2(M)$ with $Hf=-\Delta f$.

\begin{remark}\label{edcv} It is a standard fact that the intrinsic distance
	$$
	\varrho_{\mathrm{intr}}(x,y):= \sup\{f(x)-f(y): f\in W^{1,2}_\loc(M)\cap C(M), |\nabla f|\leq 1\},
	$$
	is equal to $\varrho$, where $W^{1,2}_\loc(M)$ is defined as all $\mu$-equivalence classes of Borel functions $f$ on $M$ such that $f|_U\in W^{1,2}(\mathring{U})$ for all open relatively compact $U\subset M$.
\end{remark}

The  following space will be in the center of this work:

\begin{definition} We call
	$$
	\hat{C}^\infty_c(M):=\{f\in C_c^{\infty}(M): \partial_\nu f=0 \}
	$$
	the space of \emph{Neumann cut-off functions on $M$}.
\end{definition}

Clearly, $\hat{C}^\infty_c(M)$ is dense in $L^p(M)$ for all $p\in [1,\infty)$. Defining $\hat{W}^{1,p}_0(M)$ to be the closure of $\hat{C}^\infty_c(M)$ in $W^{1,p}(\mathring{M})$, we obtain:

\begin{theorem}\label{thm D density}
	If $M$ is complete and $p\in (1,\infty)$, then one has $W^{1,p}(\mathring{M})=\hat{W}^{1,p}_0(M)$.
\end{theorem}

\begin{proof}
	As noted above, $C^{\infty}_c(M)$ is dense in $W^{1,p}(\mathring{M})$. It remains to show that any $f\in C^{\infty}_c(M)$ can be approximated by a sequence $f_n$ in $\hat{C}^\infty_c(M)$. To this end, we are going to carefully adapt the arguments which lead to the proof of Theorem 7.2.1 in \cite{Davies} to our geometric setting.\\
	First of all, using Fermi coordinates, every  $y\in \partial M$ has a chart  $(U_y,\phi_y)$ with
	$\phi_y(U_y)= D_y\times [0, \delta_y)$
	where $D_y$ is a disc centered at $0$ in $\mathbb{R}^{m-1}$, $\phi_y(y)=0$ and, for every $u\in D_y$, $\phi_y^{-1}(u,t)$ is a normal geodesic issuing from $\phi_y^{-1}(u,0)\in \partial M$. By compactness, $\mathrm{supp}\, f\cap \partial M$ is covered by finitely many sets  $\frac 12 U_i=\phi_{y_i}^{-1}(\frac 12 D_{y_i})\times [0, \delta_{y_i}/2)$, $i=1,\dots, l$. Similarly,  there exist finitely many charts $(U_i, \phi_i)$, $i=l+1, \dots, l'$ such that $U_i\cap \partial M=\emptyset$ and $\{\frac 12 U_i\}$, $i=l+1,\dots, l'$ cover $\mathrm{supp}f \setminus \cup_{i=1}^l U_i$. The usual proof shows that there exists a partition of unity $\rho_i$, $i=1, \dots, l'$, subordinate to the covering $\{U_i\}$, $i=1,\dots, l'$.
	
	Since $f=\sum_{i=1}^{l'} \rho_i f$
	it suffices to show that for every $i=i,\dots, l$ the function $\rho_if$ can be approximated by functions in $\hat{C}^\infty_c(M)$ in the $W^{1,p}$ norm.
	
	To this end, let $f$ be a $C^\infty$ function with compact support in a chart $(U_y,\phi_y)$ with the properties listed above.
	Next, let $h_n:\R\to \R$ be a smooth function with $h_n(s)=0$ if $s\leq 1/n$, $h_n(s)=s$ if $s\geq 2/n$ and $0\leq h'_n\leq 3$, and define $\tau_n:U_y\to U_y$ by
	$
	\tau_n(x)=\phi_y(y,h_n(s))$ if  $x=\phi_y(y,s)$ with $ (y,s)\in D_y\times [0,\delta_y)$.
	Finally, define $f_n(x)= f(\tau_n(x))$ and note that by construction $f_n$ is smooth, compactly supported in $U_y$ and
	\[
	f_n(\phi(y,s))=f_n(\phi(y,0))=f(y) \, \text{ if } \, y\in D_y, \, s\in[0,1/n),
	\]
	so that $f_n\in \hat{C}^\infty_c(M)$. Since $|f-f_n|^p+|\nabla(f-f_n)|^p$ is compactly supported, bounded above and converges to zero pointwise, it follows by dominated convergence that
	\[
	\left\|f-f_n\right\|_p+\left\|\nabla(f-f_n)\right\|_p\to 0 \text{ as } n\to \infty,
	\]
	as required to complete the proof.
	
\end{proof}

	

Next we are going to prove:

\begin{theorem}
	\label{thm Neumann cutoffs}
	\emph{a)} $M$ is complete, if and only if there exists a sequence of \emph{first order cut-off functions on $M$}, that is, a sequence $(\chi_n)\subset C^{\infty}_c(M)$, such that
	\begin{itemize}
		\item[(i)] $0\leq \chi_n \leq 1$,
		\item[(ii)] for each compact $K\subset M$ there exists $n_K\in \IN$ with $\chi_n\equiv 1$ for all $n\geq n_K$,
		\item[(iii)] $\left\|\nabla \chi_n \right\|_\infty\to 0$ as $n\to\infty$.
	\end{itemize}
	\emph{b)} If $M$ is complete, then there exists a sequence of \emph{first order Neumann cut-off functions}, that is, a sequence $(\chi_n)\subset C^{\infty}_c(M)$ of first order cut-off functions such that in addition
	\begin{itemize}
		\item[(iv)] $\partial_\nu \chi_n=0$.
	\end{itemize}
\end{theorem}

\begin{proof} a) Assume $M$ is complete. According to \cite[Theorem A]{giona} we can realize $M$ as a domain in a complete Riemannian manifold $N$ without boundary. It is well-known that any such $N$ admits sequences of first order cut-off functions: for example, it goes back to \cite{gaffney} that by smoothing the distance function (to a fixed reference point) one obtains a proper smooth function $f:N\to \IR$ with bounded gradient. Then $\tilde{\chi}_n(x):=\psi(f(x)/n)$, where $\psi:\IR\to [0,1]$ is smooth with compact support and equal to $1$ near $0$, does the job for $N$. Finally $\chi_n:=\tilde{\chi}_n|_M$ does the job for $N$.\\
	If conversely $M$ admits a sequence of first order cut-off functions, one can follow \cite{alberto} (proof of Theorem 2.29) to see that for a fixed $o\in M$ one has $\varrho(x,o)\to \infty$ as $x\to\infty$. We give the simple proof: Pick a compact $K\subset M$ which includes $o$ and pick $n\in\IN$ pick $n$ large enough with $\chi_n=1$ on $K$ and $\left\|\nabla \chi_n \right\|_\infty\leq 1/n$. It follows that for all $x\in M\setminus \mathrm{supp}(\chi_n)$ and all piecewise smooth curves $\gamma:[0,1]\to M$ from $o$ to $x$ one has
	$$
	1=|\chi_n(\gamma(0))-\chi_n(\gamma(1))|\leq L(\gamma)/n,
	$$
	so that taking the infimum over such curves we have shown that for all $n$ there exists a compact $K_n$ such that for all $x\in M\setminus K_n$ one has $\varrho(x,o)\geq n$.

	b) Let $\{\tilde\chi_n\}$ be a sequence of first order cut-off functions in $M$. We are going to use the idea employed in the previous theorem to modify each $\tilde \chi_n$ in such a way that the normal derivative of the modified functions $\chi_n$ vanishes and the other properties of cut-offs still hold.
	
	To achieve this, for every $n$ consider a geodesic ball $B_n$ centered at a fixed point
	$o\in \partial M$ such that $\varrho (\partial_0 B_n, \mathrm{supp}\, \tilde{\chi}_n)\ge 1$.
	Note that this implies that $\varrho(\partial M\setminus B_n, \mathrm{supp}, \tilde{\chi}_n)\geq 1$.
	
	By compactness, there exists $0<\delta_n\leq 1/2$ such that the Fermi coordinates
	$\phi_n$, which map $(y,t)$ to $\gamma_y(t)$, where $\gamma_y$ is the unit speed normal geodesic issuing from $y\in \partial M\cap B_n$, are a diffeomorphism of $(\partial M\cap B_n)\times [0,\delta_n)$ onto an open nbd 
	of $\partial M\cap B_n$ in $M$. Moreover, since $|\Id\phi_n|=1$ on $\partial M\cap B_n$, by taking a smaller $\delta_n$ we may arrange that $|d\phi_n|\leq 2$ on $(\partial M\cap B_n)\times [0,\delta_n)$.

	Next, for every $n$ let $h_n:\R\to \R $ be a smooth function such that
	\[
	h_n(s)=
	\begin{cases}
		0 &\text{if } \, s\leq \delta_n/4\\
		s &\text{if }\, s\geq \delta_n/2,
	\end{cases}
	\quad\text{and}\quad 0\leq h'_n(s)\leq 3,
	\]
	and define a map $\tau_n:B_n\to M$ by
	\[
	\tau_n(x)=
	\begin{cases}
		\phi_n(y,h_n(s)) & \text{if } \, x=\phi_n(y,s)\in U_n\\
		\tau_n(x)= x &\text{otherwise}.
	\end{cases}
	\]
	Notice that:
	\begin{itemize}
		\item[(j)] $\tau_n(x)=x$ if $\varrho(x,\partial M)\ge \delta_n/2$,
		\item[(jj) ] if $x\in B_n$ is such that $\varrho(x,\partial_0 B_n)>\delta_n$ $\partial_0$ and $\varrho(x,\partial M)<\delta_n/2$, then the minimizing geodesic joining $x$ to $\partial M$ lies enterely in $B_n$,
		\item[(jjj)] and therefore $\tau_n$ is defined and  smooth in $$U_n:= \{x\in B_n\,:\, \varrho(x,\partial_0B_n)>\delta_n\},$$ which is an open nbd of $\mathrm{supp}\,\chi_n$,
		\item[(jv)] $|\Id\tau_n|\leq 6$.
	\end{itemize}
	
	Now define $\chi_n(x) = \tilde\chi_n(\tau_n(x))$ for $x\in U_n$ and extend it with $0$ in the complement of $U_n$. The above considerations show that $\chi_n$ is smooth and that $$\mathrm{supp}(\tilde\chi_n - \chi_n)\subseteq \{x\in U_n\,:\, \varrho(x,\partial M)<\delta_n/2\},$$ so that $\chi_n$ is compactly supported. Moreover,
	\[
	\left\| \nabla \chi_n\right\|_\infty=\left\|\nabla \tilde\chi_n\circ \Id\tau_n\right\|_\infty\leq 6 \left\|\nabla \tilde\chi_n\right\|_\infty \to 0 \text{ as } n\to \infty,
	\]
	and, by construction $\partial_\nu \chi_n=0$ on $\partial M$. It remains to show that property (iii) in the statement holds. To this end, let $K$ be a compact set in $M$ and let $ B_R$ be a ball  of radius $R$ centered at $o$ such that $B_{R-1}\supseteq K$. Since $\tilde \chi_n$ is a sequence of cut-offs, there exists $n=n_R$ such that $\tilde \chi_n=1$ on $\overline{ B_R}$ for every $n\geq n_R.$
	
	We claim that for all such $n$'s one has $\chi_n=1$ on $B_{R-1}$ and therefore on $K$, as required to complete the proof. Indeed, by (j) above, $\chi_n(x)=\tilde \chi_n(x)$ if $x\in U_n$ and $\varrho(x,\partial M)\geq \delta_n/2$. It follows that $\chi_n(x)=\tilde \chi_n(x)=1$ if $x\in B_R$ and $\varrho(x,\partial M)\geq\delta_n/2.$  On the other hand, if $x\in B_{R-1} $ and $ \varrho (x,\partial M) \leq \delta_n/2$ the unique minimizing geodesic joining $x$ to $\partial M$ lies entirely in $B_R$ and since $\tilde\chi_n $ is identically equal to $1$ on its image, so is also $\chi_n$. In particular, $\chi_n(x)=1$ and the claim is proved.
\end{proof}

We note that the above construction can be used to give an alternative proof of Theorem~\ref{thm D density}. However, the given proof of Theorem~\ref{thm D density} seems to be somewhat simpler.\vspace{2mm}

\begin{remark}\label{extension} By Theorem A in \cite{giona}, $M$ can be realized as a domain in a smooth Riemannian manifold $N$ without boundary, which can be chosen complete if $M$ is so. Then:
	\begin{itemize}
		\item[$(\alpha)$] every $f\in \hat{C}^\infty_c(M)$ can be extended to a function in $C^\infty_c(N)$,
		\item[$(\beta$)]	every $f_o\in C^\infty_c(\partial M)$ can be extended to a function in $\hat{C}^\infty_c(M)$.
	\end{itemize}
	
	To see $(\alpha)$, given $f\in \hat{C}^\infty_c(M)$, let $V$ be a relatively compact open set in $\partial M$ containing $\text{supp}\, f\cap \partial M$. By compactness, the Fermi coordinates are a diffeomorphism of  $V\times (-r,r)$ onto a tubular  neighborhood $U_r$ of $V$ and it follows that there is a chart $(U_r,\varphi)$ of $N$ such that $\varphi(U_r)=D\times (-r,r)$ and  $\varphi(U_r\cap M)=D\times [0,r)$, where $D$ is a disk in $\mathbb{R}^{m-2}$, in such a way that the normal derivative corresponds to differentiation in the last variable. Using a result of R. T. Seeley, \cite{Seeley}, and multiplying by a function $h(t)$ which is $=0$ if $t<-r/2$ and $=1$ if $t>-r/4$ we extend $f\circ \varphi^{-1}$ to a smooth function defined in $\varphi(U_r)$ which lifted to $N$ provides the required extension $\tilde f$ of $f$. { Note that Seeley's construction implies that 
		there exists a constant $C$ which depends only on the the geometry of $V$ and on $r$ such that $||\tilde f||_{W^{2,2}}\leq C ||f||_{W^{2,2}}$.}
	
	The proof of $(\beta)$ is similar but easier: $f_o\in C^\infty_c(\partial M)$ and $V\subset \partial M$ is a  relatively compact neighborhood of its support, we can extend $f_o$ to a function on $N$ which is compactly supported in $U_r$ by lifting the function defined
	on $V\times (-r,r)$ by $f(u,t)=f_o(u)h(t)$, where $h(t)$ is a smooth cutoff on $\mathbb{R}$ such that $h(t)=1$ if $|t|<r/4$ and $h(t)=0$ if $|t|>r/2$.
	
\end{remark}

We go on to define $\hat{\mathscr{W}}^{2,p}_0(M)$ to be the closure of $\hat{C}^\infty_c(M)$ with respect to the norm $\left\|f\right\|_p+ \left\|\Delta f\right\|_p$.

\begin{remark} For all $f\in  \hat{\mathscr{W}}^{2,p}_0(M)$ one has $f,\Delta f\in L^p(M)$. To see this, note first that clearly $\hat{\mathscr{W}}^{2,p}_0(M)$ is a subspace of $L^p(M)$, as every sequence in $\hat{C}^\infty_c(M)$ which is Cauchy in $\hat{\mathscr{W}}^{2,p}_0(M)$ is also Cauchy in $L^p(M)$. Given $f\in  \hat{\mathscr{W}}^{2,p}_0(M)$, $\phi\in C^\infty_c(\mathring{M})$, pick a sequence $f_n\in \hat{C}^\infty_c(M)$ with $f_n\to f$ in $\hat{\mathscr{W}}^{2,p}_0(M)$. Then $\Delta f_n$ is a Cauchy sequence in $L^p(M)$, and with $h\in L^p(M)$ the limit of $\Delta f_n$ in $L^p(M)$ we have
	$$
	\left\langle f,\Delta \phi\right\rangle = \lim_n \left\langle f_n,\Delta \phi\right\rangle=  \lim_n \left\langle\Delta f_n, \phi\right\rangle= \left\langle\Delta h, \phi\right\rangle,
	$$
	and so $\Delta f\in L^p(M)$.
\end{remark}

We close this section with the following embedding result:

\begin{theorem} Let $1<p\leq 2$. Then for all $f\in \hat{\mathscr{W}}^{2,p}_0(M)$ one has
	\begin{align}\label{defg}
		\left\|\nabla f\right\|_p^2\leq 	(p-1)^{-1}\left\|\Delta f\right\|_p	\left\|f\right\|_p<\infty.
	\end{align}
	In particular, there is a continuous embedding $\hat{\mathscr{W}}^{2,p}_0(M)\subset \hat{W}_0^{1,p}(M)$, the norm of which only depends on $p$.
\end{theorem}

\begin{proof} It is sufficient to prove the asserted estimate for all $f\in \hat{C}^\infty_c(M)$. If the boundary of $M$ is empty and $M$ is complete, then this estimate can be found in \cite{cd}, and in its ultimate form with an explicit constant in \cite{honda}. If the boundary of $M$ is empty and $M$ is not complete, then by Theorem A in \cite{giona} one can realize an open relatively compact and smooth neighbourhood of the support of $f$ as a domain in a complete Riemannian manifold without boundary to obtain the result. Finally, in this way, if the boundary of $M$ is nonempty, then the estimate follows from the first part of Remark \ref{extension}.
\end{proof}

\section{Essential self-adjointness of the Neumann-Laplacian}

The main result of this section is:

\begin{theorem}\label{opo} Let $M$ be complete. Then $-\Delta$ with domain of definition $\hat{C}^\infty_c(M)$ is essentially self-adjoint in $L^2(M)$; in particular, $H$ is the unique self-adjoint extension of this operator.
\end{theorem}

To the best of our knowledge, this is the first result that deals with the essential self-adjointness of the Laplacian on a noncompact Riemannian manifold with boundary. The only mildly related result we are aware of is \cite{perez}, which deals with complex Riemannian manifolds that have a certain symmetry and that satisfy some additional analytic assumptions; there, the authors prove - with completely different methods - the essential self-adjointness of the Laplacian induced by the $\overline{\partial}$ operator.

In the proof of Theorem \ref{opo} we will use the following lemma, whose content is that the domain of $H$ is invariant under multiplication by functions in $\hat{C}^\infty_c(M)$:

\begin{lemma} \label{lemma-dom-H}
	Let $f\in\dom (H)$ and let $\chi\in \hat{C}^\infty_c(M)$. Then $f\chi\in  W^{1,2}(\mathring{M})$, $\Delta(f\chi) = \chi \Delta f + 2(\nabla f, \nabla \chi) +f \Delta \chi\in L^2(M)$ and, for all $\phi\in W^{1,2}(\mathring{M})$,
	\[
	\int_M \phi \Delta(\chi f) \,\Id\mu = -\int_M (\nabla \phi , \nabla (\chi f)) \,\Id\mu.
	\]
	It follows that  $\chi f\in \dom (H)$.
\end{lemma}
\begin{proof}[Proof of  Lemma~\ref{lemma-dom-H}]
	Clearly $\chi f\in W^{1,2}(\mathring{M})$ and a computation shows that, for every  $\phi\in C^\infty_c(\mathring{M})$,
	\[
	\begin{split}
		\int_M \chi f \Delta \phi \,\Id\mu &=- \int_M (\nabla \phi, \chi\nabla f +f \nabla \chi) \,\Id\mu \\
		&= -\int_M (\nabla(\chi \phi),\nabla f)\,\Id\mu + \int_M\phi(\nabla \chi, \nabla)\,\Id\mu +  \int_M \phi \operatorname{div} (f\chi) \,\Id\mu\\
		&= \int_M \phi [\chi\Delta f + 2(\nabla \chi, \nabla f) +f\Delta \chi] \,\Id\mu,
	\end{split}
	\]
	where we have used that $\chi \phi \in W^{1,2}(\mathring{M}),$   $f\in \dom(H)$ and that $f\nabla\chi$ is a  $W^{1,2}$ vector field with compact support in $\mathring{M}.$ Thus
	\[
	\Delta(\chi f)= \chi\Delta f + 2(\nabla \chi, \nabla f) +f\Delta \chi \in L^2(M).
	\]
	Next let  $\phi\in W^{1,2}(\mathring{M})$; using the fact that $\chi f\in W^{1,2}(\mathring{M})$, $f\phi\in W^{1,1}(\mathring{M})$ and that $\partial_\nu \chi=0$ we obtain
	\[
	\begin{split}
		\int_M \phi\Delta(\chi f) \,\Id\mu&= 
		\int_M (\chi\phi) \Delta f + 2(\nabla f, \phi \nabla \chi ) +f \phi\Delta \chi] \,\Id\mu\\
		&=\int_M [-(\nabla(f, \nabla(\chi \phi)+ 2(\nabla f, \phi\nabla \chi) - (\nabla(f\phi),\nabla \chi)]\,\Id\mu +\int_{\partial M} f\phi\partial_\nu \chi \,\Id\mu\\
		&= -\int_M (\nabla \phi, \nabla (\chi f))\,\Id\mu,
	\end{split}
	\]
	as required to conclude the proof of the lemma.
\end{proof}

\begin{proof}[Proof of Theorem~\ref{opo}] We are going to show that $\hat{C}^\infty_c(M)$ is an operator core for $H$. To this end, our proof will follow Chernoff's philosophy \cite{chernoff} of using the wave equation for proving essential self-adjointness (see also \cite{post, kolb}). However, in the presence of boundary, some additional care has to be taken in order obtain enough regularity for functions in the intersection of the spaces $\dom(H^k)$.\vspace{1mm}
	
	Step 1: For all $t>0$, $f\in L^2(M)$ one has $\mathrm{e}^{-tH}f\in C^\infty(M)$ with $\partial_\nu(\mathrm{e}^{-tH}f)=0$.\vspace{1mm}
	
	Proof of step 1: By spectral calculus we have
	$$
	h:=\mathrm{e}^{-tH}f\in \bigcap^\infty_{k=0} \dom(H^k)\subset \{f\in W^{1,2}(\mathring{M}) :\Delta^k f\in L^2(M)\quad\text{for all $k$}\},
	$$
	and it satisfies the weak Neumann boundary condition
	$$
	\left\langle \Delta h,\phi\right\rangle= \left\langle \nabla h,\nabla \phi\right\rangle\>,
	$$
	for all $\phi\in W^{1,2}(\mathring{M})$.
	
	Clearly $h$ is smooth in $\mathring{M}$ due to local interior elliptic regularity and Sobolev embedding. We need to prove that $h$ is smooth up to the boundary and that $\partial_\nu h=0$. So we fix an arbitrary $x\in \partial M$ and pick a { relatively compact coordinate neighbourhood $N'\subset M$ of $x$ such that $N:=\overline{N'}$ is a smooth manifold with boundary. Let also $V\Subset N'$ be a neighborhood of $x$ and let $\chi, \rho \in C^\infty(M)$ by such that $\chi=1$ on $V$, $\rho=1$ on $\text{supp\,}\chi $ and $\text{supp}\, \rho \Subset N'$. 
		Since $h\in \dom(H)$,  by Lemma~\ref{lemma-dom-H},
		$$
		h_0:=\chi h \in W^{1,2}(\mathring{M}),\quad \Delta h_0=:h_1\in  L^2(M),
		$$
		and the weak Neumann boundary condition holds
		$$
		\left\langle \Delta h_0,\phi\right\rangle= \left\langle \nabla h_0,\nabla \phi\right\rangle\quad\text{for all $\phi\in W^{1,2}(\mathring{M})$}.
		$$
		Thus, given $\phi_0\in W^{1,2}(\mathring{N})$, $\rho\phi_0\in W^{1,2}(\mathring{M})$ and we deduce that
		\[
		\int_N \phi_0 h_1 \,\Id\mu = \int_M\Delta h_0 (\rho\phi_0)\,\Id\mu =-\int_M (\nabla h_0,\nabla(\rho\phi_0))\,\Id\mu = \int_N(\nabla h_0, \nabla\phi_0) \,\Id\mu.
		\]
		with $h_1=\Delta h_0\in L^2(N)$
		
		It follows from \cite[Theorem 5.9]{Lieberman} that $h_0\in W^{2,2}(\mathring{N})$. } Bootstrapping this argument leads to $h_0\in W^{k,2}(\mathring{N})$ for all $k$, and so $h_0\in C ^\infty(N)$ by Sobolev embedding and so $h\in C ^\infty(M)$. Finally, the weak Neumann boundary condition and Green's formula in combination with the second part of Remark \ref{extension} give $\partial_\nu h=0$. \vspace{1mm}
	
	Step 2: The space
	$$
	\dom_c(H):= \{f\in \dom(H): \text{$f$ has a compact support}\}
	$$
	is an operator core for $H$, that is, $\dom_c(H)$ is dense in the norm $\left\|f\right\|_{H}:=\left\|f\right\|+\left\|H f\right\|$ in $\dom(H)$. \vspace{1mm}
	
	Proof of step 2: Note first that by Remark \ref{edcv}, Theorem \ref{aux1} and Theorem \ref{aux3} we get that for all $U_1,U_2\subset M$ open and disjoint, $f_1,f_2\in L^2(M)$ with $\mathrm{supp}(f_i)\subset U_i$, $i=1,2$, and all $0<s<\varrho(U_1,U_2)$, one has
	\begin{align}
		\langle\cos(s\sqrt{H})f_1,f_2  \rangle =0.\label{ga2}
	\end{align}
	The latter fact implies that if $\mathrm{supp}(f)\subset \IB(x_0,r)$ for some $x_0\in M$, $r>0$, then for any $t>0$ one
	has
	\begin{align}
		\mathrm{supp}(\cos (t\sqrt{H})f)\subset
		\overline{\IB(x,r+t)},
	\end{align}
	and this closed ball is compact due to the completeness of $M$ (since completeness on $M$ is equivalent to the Heine-Borel property of $M$ \cite{giona}). Now the claim follows from an abstract functional analytic fact (cf. Theorem \ref{aux2} below).\vspace{1mm}

	Step 3: $\hat{C}^\infty_c(M)$ is dense in $\dom_c(H)$ in the norm $\left\|\cdot\right\|_{H}$.\\
	Proof of step 3: let $f\in \dom_c(H)$. Pick a function $\chi\in C^{\infty}_c(M)$ with $\partial_\nu \chi=0$ and $\chi= 1$ on $\mathrm{supp}(f)$. By step 1 we have $\chi e^{-tH}f\in \hat{C}^\infty_c(M)$ and
	$$
	\left\|\chi e^{-tH}f-f\right\|=\left\|\chi \left(e^{-tH}f-f\right)\right\|\to 0
	$$
	as $t\to 0+$. Moreover
	\begin{align*}
		H(\chi e^{-tH}f-f)  & =  H\left[\chi \left(e^{-tH}f- f\right)\right]\\
		&=\chi\left( H  \e^{-t H} f -Hf\right)- 2(\nabla \chi,\nabla (\e^{-t H} f-f)) - (H\chi )(\e^{t H} f -f),
	\end{align*}
	which shows that $\left\|H (\chi e^{-tH}f-f)\right\|\to 0$ as $t \to 0+$, completing the proof.
\end{proof}

\begin{example} Let us show that one cannot drop completeness in Theorem \ref{opo}. Pick $M:=[1,\infty)$, $\psi:=1$, and fix the Riemannian metric on $M$  as  $g(x)\coloneqq 1/x^4$. Then, the metric distance between two points $p,q$ is given by  $\varrho(p,q) = \operatorname{sgn}(q-p)(1/p - 1/q)$ and $M$ is not complete (the sequence $p_n=n$ is Cauchy but does not converge to any point in $M$). The Laplace-Beltrami operator is given in this case by
	\begin{equation*}
		-\Delta f= -x^2\frac{\Id}{\Id x}\left(x^2 f\right).
	\end{equation*}
	Chose $c_1$ and $c_2$ such that
	$$
	\hat{f}(x) \coloneqq  c_1\cosh(1/x) + c_2 \sinh(1/x)
	$$
	satisfies $\frac{\Id\hat{f}}{\Id x}_{|x=1}=0$. It is straightforward to check that $-\Delta \hat{f} = -\hat{f}$. By the boundedness of $\hat{f}$, it follows that $\hat{f}, \Delta \hat{f} \in L^2(\mathring{M})$, and by the Neumann boundary condition at $x=1$ it follows that
	$$
	\langle -\Delta h, \hat{f}\rangle = \langle  h, -\Delta\hat{f}\rangle\quad\text{ for every $h\in \hat{C}^\infty_c(M)$}.
	$$
	That is, for $H_0$ defined as $-\Delta$ with domain of definition $\hat{C}^\infty_c(M)$ one has $\hat{f}\in \dom(H^*_0)$ and $H_{0}^*\hat{f}= -\Delta \hat{f}$. As a consequence,  $(H_{0}+1)^*\hat{f} = 0$. Therefore, $(H_{0}+1)^*=H_{0}^*+1$ has a nontrivial kernel and we conclude that $H_{0}$ is not essentially self-adjoint in $L^2(M)$.
\end{example}

\vspace{1mm}

\emph{Acknowledgements:} The authors would like to thank Diego Pallara, Stefano Pigola and Giona Veronelli for very helpful discussions.

\section{Appendix}

Let $M$ be locally compact seperable metrizable space with $\mu$ a Radon measure on $M$ with full support.

\begin{theorem}\label{aux1} Let $H$ be a self-adjoint nonnegative operator in $L^2(M,\mu)$. Then the following two conditions are equivalent for every distance $\varrho$ on $M$ which induces the original topology:
	\begin{itemize}
		\item[(i)] There are constants $C\geq 1$, $a>0$ such that for all $t>0$, all open $U_1,U_2\subset M$, and all $f_1,f_2\in L^2(M,\mu)$ with $\mathrm{supp}(f_i)\subset U_i$, $i=1,2$, one has
		\begin{align}\label{dg}
			\left|\left\langle \mathrm{e}^{-t H}f_1,f_2 \right\rangle\right|\leq C\mathrm{e}^{a t}\mathrm{e}^{-\varrho(U_1,U_2)^2/(4t)}\left\|f_1\right\|\left\|f_2\right\|.
		\end{align}
		
		\item[(ii)] For all open $U_1,U_2\subset M$, and all $f_1,f_2\in L^2(M,\mu)$ with $\mathrm{supp}(f_i)\subset U_i$, $i=1,2$, one has
		$$
		\left\langle \mathrm{cos}\big(s\sqrt{H}\big)f_1,f_2 \right\rangle =0\>\text{  for all $0<s<\varrho(U_1,U_2)$.}
		$$
	\end{itemize}
\end{theorem}

\begin{proof} This is Theorem 3.4 in \cite{sikora}.
\end{proof}

Note that property (i) above is usually referred to as \emph{Davies-Gaffney bound}, while (ii) is referred to as \emph{finite wave propagation speed}.\vspace{1mm}

Before we state the next theorem, we recall given a self-adjoint operator $H\geq 0$ in $L^2(M,\mu)$ and $t\in \IR$ one defines the unitary operator $\cos(t\sqrt{H})$ via spectral calculus. Then $\cos(t\sqrt{H})$ preserves $\dom(H)$, and for all $f\in \dom(H)$ the function $f(t):=\cos(t\sqrt{H})$ solves the abstract wave equation $f''(t)=-H f(t)$.

\begin{theorem}\label{aux2} Let $H\geq 0$ be a self-adjoint operator in $L^2(M,\mu)$, and assume furthermore that the compactly supported elements $\dom_c(H)$ of $\dom(H)$ are dense in $L^2(M,\mu)$ and that for all $t>0$ one has the mapping property
	$$
	\cos(t\sqrt{H}):\dom_c(H)\longrightarrow \dom_c(H).
	$$
	Then $\dom_c(H)$ is an operator core for $H$.
\end{theorem}

\begin{proof} See (the proof of) Theorem 3 in \cite{kolb}.
\end{proof}

\begin{theorem}\label{aux3}  Let $\mathscr{E}$ be a strongly local regular Dirichlet form in $L^2(M,\mu)$, with $H\geq 0$ the associated self-adjoint operator, and with $\Gamma$ the associated energy measure. Assume also that the intrinsic distance
	$$
	\varrho_{\mathrm{intr}}(x,y):= \sup\{f(x)-f(y): f\in \dom_\loc(\mathscr{E})\cap C(M), \Id\Gamma(f)\leq \Id\mu\}
	$$
	induces the original topology on $M$. Then $H$ satisfies (\ref{dg}) with $C=1$ and $a=0$.
\end{theorem}

\begin{proof} This is Theorem 0.1 in \cite{sturm2}. Note that, in general, $\varrho_{\mathrm{intr}}$ is only a pseudo metric, that is, it can attain the value $+\infty$ or may be degenerate (but nevertheless induces a topology).
\end{proof}

\end{document}